\documentclass[12pt,reqno]{amsart}
\usepackage{amsthm, amsfonts,  amsmath, amssymb}
\usepackage{url}
\usepackage[a4paper, top = 1.2in, bottom = 1.2in, left = 1.2in, right = 1.2in]{geometry}
\newtheorem{thm}{Theorem}[section]

\newtheorem{prop}[thm]{Proposition}
\theoremstyle{definition}
\newtheorem{dfn}[thm]{Definition}
\newtheorem{Hypo}[thm]{Hypothesis}

\newtheorem{conj}[thm]{Conjecture}

\newtheorem{remark}[thm]{Remark}
\theoremstyle{plain}

\numberwithin{equation}{section}

\newcommand{\N}{\mathbb{N}}

\newcommand{\mbP}{\mathbb{P}}
\newcommand{\R}{\mathbb{R}}

\newcommand{\Z}{\mathbb{Z}}

\newcommand{\SL}{\mathrm{SL}}
\newcommand{\new}{\mathrm{new}}

\newcommand{\lra}{\longrightarrow}
\newcommand{\ra}{\rightarrow}

\newcommand{\mrm}[1]{\mathrm{#1}}

\title[Multiplicity one theorems]{A Variant of Multiplicity one theorems for half-integral weight modular forms}
\author[N. Kumar]{Narasimha Kumar}
\email{narasimha.kumar@iith.ac.in}
\address{
Department of Mathematics \\
Indian Institute of Technology Hyderabad\\
Kandi, Sangareddy - 502285\\
INDIA. 
}
\date{}
\begin{document}
\begin{abstract}
We show that signs of Fourier coefficients, on certain sub-families, determine the half-integral weight cuspidal eigenform uniquely, 
up to a positive constant. 
We also study sign change results for the product of the Fourier coefficients of two distinct
half-integral weight eigenforms.
\end{abstract}
\maketitle
\section{introduction}
Determination of modular forms is one of the fundamental and an interesting problem in number theory. 
One can determine cuspidal Hecke eigenforms of integral weight by the central critical values of the corresponding $L$-functions twisted 
by certain Dirichlet characters or by a family of modular forms (cf.~\cite{LR97},~\cite{GHS09}). On the other hand, the eigenvalues of the Hecke operators at primes 
$p$ acting on the space of newforms  will also determine the newform uniquely.  In literature, these are known as multiplicity one theorems.
% and there are several refinements of these also exists. 
% For example, Ramakrishnan's refinement of the strong multiplicity one for $\GL(2)$ which states that if two cuspidal automorphic representations 
% of $\GL(2)$ over a global field have isomorphic local components outside a set of primes of density less than $1/8$, 
% then they are globally isomorphic (cf.~\cite{Ram94}).

For primitive forms $f$ of integral weight, the extent to which the signs of Hecke eigenvalues at primes $p$ determine $f$ uniquely has been 
first studied by Kowalski et al.~\cite{KLSW10}
(and also by Matom\"aki~\cite{Mat12}, who refined some of their results). A natural question to ask 
if similar results continue to hold in the case of half-integral weight modular forms? 
Classically, there are several multiplicity one theorems available in the literature for half-integral weight cuspidal eigenforms in terms of their Hecke eigenvalues (cf.~\cite{MRV90},~\cite{Koh82}). 

In this article, we show that signs of Fourier coefficients on certain sub-families, which are accessible via the Shimura correspondence, determine the half-integral weight cuspidal eigenform uniquely, 
up to a positive scalar multiple (cf. \S 3). We also study the sign change results for the product of  Fourier coefficients of two distinct half-integral weight cuspidal eigenforms, 
by assuming pair Sato-Tate conjecture for their corresponding Shimura lifts. We also state a equi-distribution version of this result (cf. \S 4). The basic idea of this article comes 
from the author's previous article~\cite{Kum15}, which in turn a variant of the techniques of ~\cite{AIW15}.  

% In \S\ref{mainresult}, we show that signs of Fourier coefficients determine the half-integral weight cuspidal eigenform uniquely.  

\section{Preliminaries} 
\label{preliminaries}
Let $\mathbb{P}$ denotes the set of all prime numbers.
Now, we let us recall the Sato-Tate measure and the notion of natural density and analytic density for subsets of $\mathbb{P}$.

\begin{dfn}
	The Sato-Tate measure $\mu_{\mrm{ST}}$ is the probability measure on  $[-1,1]$ given by $\frac{2}{\pi} \sqrt{1-t^2} dt$. 
\end{dfn}
\begin{dfn}
	Let $S$ be a subset of  $\mathbb{P}$. The set $S$ has  natural density $d(S)$ (resp., analytic density $d_{\rm{an}}(S)$),
	if the limit
	\begin{equation}
	\underset{x \ra \infty}{\mrm{lim}}\  \frac{\# \{ p \leq x: p\in S\}}{\pi(x)} \quad \bigg(\mathrm{resp.,} \ \ \underset{s \ra 1^{+}}{\mrm{lim}}\  \frac{\sum_{p \in S} \frac{1}{p^s}}{\log(\frac{1}{s-1})} \bigg)
	\end{equation}
	exists and is equal to $d(S)$ (resp., is equal to  $d_{\rm{an}}(S)$), where $\pi(x):= \# \{ p \leq x: p\in \mathbb{P}\}$. 
\end{dfn}

\begin{remark}
	If a subset $S \subseteq \mathbb{P}$ has a natural density, then it also has an analytic density, and the two densities are the same. 
	Observe, if $|S|<\infty$, then $d(S)=0$  and hence  $d_{\mathrm{an}}(S)=0$.
\end{remark}

Let $k,N$ be natural numbers and $\chi$ be a Dirichlet character modulo $4N$. Then $S_{k+\frac{1}{2}}(4N,\chi)$ be  
the space of cusp forms of weight $k+\frac{1}{2}$, level $4N$ with character $\chi$. We let $\chi_0$ to denote the trivial character.
When $k=1$, we shall work only with the orthogonal complement (with respect to the Petersson inner product) of the subspace of $S_{k+\frac{1}{2}}(4N,\chi)$ spanned by single-variable unary theta functions.

Let $N \geq 1 $ be an odd and square-free integer. Let $S^\new_{k+\frac{1}{2}}(4N,\chi)$ denote the space of newforms inside $S_{k+\frac{1}{2}}(4N,\chi)$. Let $S_{k+\frac{1}{2}}^{+}(4N,\chi)$ denote Kohnen's $+$-subspace of $S_{k+\frac{1}{2}}(4N,\chi)$	
consisting of modular forms $f = \sum_{n=1}^{\infty} a_f(n) q^n $ with
$a_f(n)=0$  for $n \equiv 2,(-1)^{k+1} \pmod 4$. Let $S^{+,\new}_{k+\frac{1}{2}}(4N,\chi)$ denote the space of newforms in $S^+_{k+\frac{1}{2}}(4N,\chi)$.

\section{Multiplicity one theorem}
\label{mainresult}
In this section, we shall state one of the main result of this article and shall give a proof of it. 
Throughout this article, we shall stick to the following notation.

\begin{Hypo}
\label{key1}
Let $f=\sum_{n=1}^{\infty} a(n) q^n \in S^\new_{k_1+\frac{1}{2}}(4N_1,\chi_0)$ ($g=\sum_{n=1}^{\infty} b(n) q^n \in S^\new_{k_2+\frac{1}{2}}(4N_2,\chi_0)$)
be a non-zero cuspidal eigenform for operators $T_{p^2}$ for primes $p \nmid 2N_1$, (resp., $p \nmid 2N_2$), where  $k_1,k_2 \geq 1$ are integers
and $N_1,N_2$ are odd and square-free integers. Suppose there exists a square-free integer $t \geq 1$ such that $a(t)b(t) \neq 0$.
Suppose $F_t =\sum_{n=1}^{\infty} A_t(n) q^n \in S^\new_{2k_1}(2N_1)$ ($G_t=\sum_{n=1}^{\infty} B_t(n) q^n \in S^\new_{2k_2}(2N_2))$ 
are the cuspidal eigenforms for operators $T_p$  for primes  $p\nmid 2N_1$ (resp., $p\nmid 2N_2$) corresponding to $f$ (resp., $g$) 
under the Shimura lift for the square-free integer $t$. 
\end{Hypo}

WLOG, we can assume that $a(t)=1$, $b(t)=1$. This is because, once we have the theorem in this case, we can apply the theorem with
$\frac{f}{a(t)}$ and $\frac{g}{b(t)}$ to prove the general case. An advantage of this reduction is that, in this case, the eigenforms $F_t$, $G_t$ are 
become primitive forms and they are independent of $t$. For simplicity, we denote them by $F,G$, and their coefficients with $A(n),B(n)(n \in \N)$ respectively.
Observe that, the levels $2N_1,2N_2$ are square-free integers, hence the primitive forms $F,G$ are without complex multiplication.

In this article, we shall follow this notation: 
Let $f,g,F,G$ be as in Hypothesis~\ref{key1}. For any prime $p$,
Let $C(p) \in [-1,1]$, $D(p)\in [-1,1]$ denote $\frac{A(p)}{2p^{k_1-\frac{1}{2}}}$,  $\frac{B(p)}{2p^{k_2-\frac{1}{2}}}$ for $F$,$G$, resp.,
Now, we are ready to state one of the main result of this article.
\begin{thm}[Multiplicity one theorem]
\label{multiplicationthm1}
Let $f,g$ be two half-integral weight eigenforms satisfying Hypothesis~\ref{key1}.
If $a(tp^2)$ and $b(tp^2)$ have the same sign for every $p \not \in E_0$ with  $d_{\mrm{an}}(E_0) \leq 6/25$
then $N_1=N_2$, $k_1=k_2$, and $f= g$, up to a positive scalar multiple.
\end{thm}
Before we proceed to prove the theorem, we recall some basic properties of Shimura lift and also will prove a proposition
which will be useful.

Let $f$ be a cuspidal eigenform as in Hypothesis~\ref{key1}  and $F$ denote the Shimura lift corresponding to $f$. 
By ~\cite[\S 5]{CR94}, we can see the relation between the Fourier coefficients of $f$ and those of its lift $F$, namely 
\begin{equation}
\label{keyequation}
A(n) = \sum_{d|n, (d,2N_1)=1} \chi_{1}(d)d^{k_1-1} a\left( \frac{tn^2}{d^2}\right), 
\end{equation}
where $\chi_1(d)$ is a quadratic character, whose explicit expression is not necessary in our  context. 
In particular, for a prime $p$ with $(p,2N_1)=1$, the relation becomes
\begin{equation}
\label{keyequation1}
 a(tp^2)  =  A(p)- \chi_1(p) p^{k_1-1}.
\end{equation}
\begin{remark}
Since $F$ is a primitive form with trivial nebentypus, one knows that the  Fourier coefficients $A(n)(n \in \N)$ of $F$ are real numbers. 
In particular, by~\eqref{keyequation1}, we can see that $\{ a(tp^2) \}_{p \in \mathbb{P}}$ are also real numbers, hence we can talk about signs and sign changes.
\end{remark}  
By~\eqref{keyequation1}, for any $p \in \mathbb{P}$, we have that
\begin{equation*}
a(tp^2)<0 \Longleftrightarrow  C(p) < \frac{\chi_1(p)}{2\sqrt{p}}.
\end{equation*}

We see that, if $a(tp^2)$ is negative, then it does not mean that $C(p)$  is negative. So, Theorem~\ref{multiplicationthm1} is not an immediate
consequence of a theorem of Matom\"aki (cf. \cite[Theorem 2]{Mat12}). However, we can still deduce our theorem from there by a trick,
which is the content of the following proposition.

\begin{prop}
\label{keyproposition1}
Let $f=\sum_{n=1}^{\infty} a(n) q^n \in S^\new_{k_1+\frac{1}{2}}(4N_1,\chi_0)$ be as in Hypothesis~\ref{key1}.
The natural density of primes $p$ for which $a(tp^2)$ and $C(p)$ have the opposite sign is zero.
\end{prop}
\begin{proof}
	To prove the proposition, it is sufficient to show that
	$$d\left(\left\lbrace p\ \mrm{prime}: p \nmid 2N_1,  \frac{1}{2 \sqrt{p}} > C(p) \geq 0 \right\rbrace\right)=0.$$
	
	For any fixed (but small) $\epsilon >0$, we have the following inclusion of sets 
	$$ \left\lbrace p\leq x: p \nmid 2N_1,\  C(p)\in[0,\epsilon] \right\rbrace \supseteq 
	\left\lbrace  p \leq x: p \nmid 2N_1,\ p> \frac{1}{4\epsilon^2},\ 0 \leq C(p) <\frac{1}{2\sqrt{p}} \right\rbrace.$$ 
	Hence, we have
	$$ \# \{p\leq x: p \nmid 2N_1,\  C(p)\in[0,\epsilon]\} + \pi\left(\frac{1}{4\epsilon^2}\right) \geq \# \{ p \leq x:  
	p \nmid 4N_1,\  0 \leq C(p) < \frac{1}{2\sqrt{p}}\}.$$
	Now divide the above inequality by $\pi(x)$
	$$ \frac{\# \{p\leq x:  p \nmid 2N_1,\  C(p)\in[0,\epsilon]\}}{\pi(x)} + \frac{\pi\left(\frac{1}{4\epsilon^2}\right) }{\pi(x)} 
	\geq \frac{\# \left\lbrace p \leq x: p \nmid 2N_1,\ 0 \leq C(p)  < \frac{1}{2\sqrt{p}} \right\rbrace}{\pi(x)}.$$
	The term $\frac{\pi(\frac{1}{4\epsilon^2})}{\pi(x)}$ tends to zero as $x\ra \infty$ as $\pi(\frac{1}{4\epsilon^2})$ is finite.
	By the Sato-Tate equi-distribution theorem (\cite[Thm. B.]{BGHT11}), we have
	$$ \frac{\# \{ p\leq x : C(p)\in[0,\epsilon] \} }{\pi(x)}  \lra \mu_{\mrm{ST}}([0,\epsilon]) \quad\ \mrm{as}\quad 	x \ra \infty. $$ 
	This implies that 
	\begin{equation}
	\label{key-inequality-1}
	\underset{x \ra \infty}{\mrm{lim\ sup}} \ \frac{\{p\leq x: p \nmid 2N_1, 0 \leq C(p)  < \frac{1}{2 \sqrt{p}} \}}{\pi(x)} \leq \mu_{\mrm{ST}}([0,\epsilon ]).
	\end{equation}
	Since the inequality~\eqref{key-inequality-1} holds for all $\epsilon>0$, we have that 
	$$ \underset{x \ra \infty}{\mrm{lim}} \ \frac{\{p\leq x : p \nmid 2N_1, 0 \leq C(p)  < \frac{1}{2 \sqrt{p}} \}}{\pi(x)} = 0.$$
   The proof in the other case, i.e., $a(tp^2)$ is positive, is similar to the above one.
\end{proof}

Now, we are ready to prove Theorem~\ref{multiplicationthm1}.
\begin{proof}
By Proposition~\ref{keyproposition1},  the signs of $a(tp^2)$ and $C(p)$ are exactly the same, except possibly for a natural density zero set of primes, say $E_f$. 
Similarly, for the eigenform $g$ and denote the set by $E_g$.  Take $E=E_0 \cup E_f \cup E_g \subseteq \mathbb{P}$.
 
Since $a(tp^2)$ and $b(tp^2)$ have same sign for every $p \not \in E$ with  $d_{\mrm{an}}(E) \leq 6/25$, then $C(p)$ and $D(p)$ also have same sign for every $p \not \in E$. 
This implies that, $C(p)$ and $D(p)$ have same sign for every prime $p \not \in E$ with  analytic density $\leq 6/25$. 
This implies that $k_1=k_2$, $N_1=N_2$ and $F=G$,  by a theorem of Matom\"aki (cf.~\cite[Theorem 2]{Mat12}). 

Since Shimura lift commutes with the Hecke operators, we see that,
for $p \nmid 2N_1N_2$, the $T_{p^2}$-eigenvalues of $f$,$g$ are the same, since they coincide with the $T_p$-eigenvalue of $F(=G)$.
By ~\cite[Theorem 5]{MRV90}, we see that the half-integral weight cuspidal eigenforms $f$ is a scalar multiple of $g$. 

In the general case, we can apply the proof with $f/a(t)$ and $g/b(t)$ to prove that $f$ is a scalar multiple of $g$. 
Since the $tp^2$-th coefficients of $f.g$ have the same sign, this shows that the constant has to be positive.
\end{proof}
We remark that, the above theorem is also true for eigenforms in the Kohnen's +-space $f=\sum_{n=1}^{\infty} a(n) q^n \in S^{+,\new}_{k_1+\frac{1}{2}}(4N_1,\chi_0)$,
$g=\sum_{n=1}^{\infty} b(n) q^n \in S^{+,\new}_{k_2+\frac{1}{2}}(4N_2,\chi_0)$ with $k_1 \equiv k_2 \pmod 2$.
In this case, the Shimura lifts, corresponding to a fundamental discriminant $D$, of $f,g$ 
belong to $S^\new_{2k_1}(N_1)$, $S^\new_{2k_2}(N_2)$, resp., (cf.~\cite[\S5]{CR94},~\cite{Koh82}).   Now,  the rest of the proof is similar to the proof of Theorem~\ref{multiplicationthm1}.

\section{Equi-distribution result for the product of Fourier coefficients}
\label{variant}
Recently, a variant of sign change result for the product of  Fourier coefficients of half-integral weight weight modular eigenforms has been studied in~\cite{GKR15}. 
They show that there exists an infinite set $S \subset \mathbb{P}$, such that
for any prime $p\in S$, the sequence $\{a(tp^{2m}) b(tp^{2m})\}(m \in \N)$ change signs infinitely often.

In this section, we shall  study the sign change results for the product of Fourier coefficients $\{a(tp^{2}) b(tp^{2})\}(p \in \mathbb{P})$.
In fact, we prove an equi-distribution result for the product of Fourier coefficients by assuming the pair Sato-Tate conjecture for non-CM Hecke eigenforms 
of integral weight (Conjecture~\ref{Sato-Tate-conj} below). Now, let us recall the pair Sato-Tate equi-distribution conjecture.
\subsection{Pair Sato-Tate equi-distribution conjecture:}
For $i=1,2$, let $g_i = \sum_{n=1}^{\infty} b_i(n)q^n$ be primitive eigenforms of weight $2k_i$ and level $2N_i$, resp.,  
For $i=1,2$, by Deligne's bound, for any prime $p$, we have that 
$$ |b_i(p)| \leq 2 p^{k_i-\frac{1}{2}}, $$
and we let 
\begin{equation}
\label{Sato-Tate-normalization}
B_i(p):=\frac{b_i(p)}{2p^{k_i-\frac{1}{2}}} \in [-1,1].
\end{equation}
We have the following pair Sato-Tate equi-distribution conjecture for the pair $(g_1,g_2)$. 
\begin{conj}
\label{Sato-Tate-conj}
Let $g_1,g_2$ be distinct non-CM primitive forms of weight $2k_1,2k_2$ and level $2N_1,2N_2$, resp.,
Assume that they are not twists of each other.  For any two subintervals $I_1\subseteq [-1,1], I_2 \subseteq [-1,1]$, we have 
$$ d(S(I_1,I_2)) = \underset{x \ra \infty}{\mathrm{lim}}  \ \frac{\#S(I_1,I_2)(x)}{\pi(x)} = \mu_{\mrm{ST}}(I_1) \mu_{\mrm{ST}}(I_2)= 
\frac{4}{\pi^2} \int_{I_1}  \sqrt{1-s^2}  ds  \int_{I_2}  \sqrt{1-t^2}  dt,$$
where 
\begin{align*}
S(I_1,I_2)     &=  \left\lbrace p \in \mathbb{P}: p\nmid 2N_1N_2, B_1(p) \in I_1, B_2(p) \in I_2 \right\rbrace \\
S(I_1,I_2)(x) &=  \left\lbrace p\leq x: p\in S(I_1,I_2) \right\rbrace.
\end{align*}
In other words, the Fourier coefficients at primes are independently distributed with respect to the Sato-Tate distribution.
%, when $p$ runs through the primes not dividing $N$.
\end{conj}

\subsection{Results for the product of the Fourier coefficients:}
Let $f,g$ be two half-integral weight eigenforms as in Hypothesis~\ref{key1}. As before, we shall assume that $a(t)=b(t)=1$ (cf. Remark~\ref{generaltcase}).
For the notational convenience, we let 
$$\mathbb{P}_{<0}:=\{ p \in \mathbb{P}: p\nmid 2N_1N_2,\ a(tp^2)b(tp^2)<0 \},$$ and 
similarly $\mathbb{P}_{>0}$, $\mathbb{P}_{\leq 0}$, $\mathbb{P}_{\geq 0}$, 
and $\mathbb{P}_{=0}$.  We let
$$\pi_{<0}(x):= \# \{p\leq x: p \in \mathbb{P}_{<0}\},$$ and similarly
$\pi_{>0}(x), \pi_{\leq 0}(x)$, $\pi_{\geq 0}(x)$, and $\pi_{=0}(x)$.

\begin{thm}
\label{main-thm-2}
Let $f,g$ be two distinct  half-integral weight eigenforms as in Hypothesis~\ref{key1}. Assume that $F, G$ are not twists of each other and the pair Sato-Tate conjecture holds for $(F,G)$. 
Then the product of Fourier coefficients $\{a(tp^2)b(tp^2)\}(p \in \mbP)$ change signs infinitely often.  
Moreover, the sets $$\mathbb{P}_{>0}, \mathbb{P}_{<0}, \mathbb{P}_{\geq 0}, \mathbb{P}_{\leq 0}$$ have natural density $1/2$, and $d(\mathbb{P}_{=0}) = 0$.
\end{thm}
\begin{proof}
First, we observe that the Shimura lifts $F,G$ are primitive forms and without complex multiplication (CM). This is because there are no newforms with CM for square-free levels
and $2N_1, 2N_2$ are square-free.   Then, by Shimura correspondence, for any prime $(p,2N_1N_2)=1$,
we have that
$$a(tp^2)  =  A(p)- \chi_1(p) p^{k_1-1},$$ 
where $\chi_1$ is a quadratic character. Similarly, for $g$,
%(p)=\chi_0(p)\left( \frac{(-1)^{k_1}4N_1^2t}{p} \right),$ and
$$b(tp^2)  =  B(p)- \chi_2(p) p^{k_2-1},$$ 
where $\chi_2$ is a quadratic character.
%$(p)= \chi_0(p)\left(\frac{(-1)^{k_2}4N_2^2t}{p}\right).$

From the above equation, we get that
\begin{equation*}
a(tp^2)>0 \Longleftrightarrow   1 \geq C(p) > \frac{\chi_1(p)}{2 \sqrt{p}}, \quad a(tp^2)<0 \Longleftrightarrow \frac{\chi_1(p)}{2 \sqrt{p}} > C(p)\geq -1. 
\end{equation*}
Similar inequalities also hold for the Fourier coefficients $\{b(tp^2)\}_{p \in \mathbb{P}}$ as well.
First, we shall show that
$$\underset{x \ra \infty}{\mrm{lim\ inf}} \ \frac{\pi_{< 0}(x)}{\pi(x)} \geq \mu_{\mrm{ST}}([0,1])=\frac{1}{2}.$$
% Similarly,  $\underset{x \ra \infty}{\mrm{lim\ inf}} \ \frac{\pi_{\leq 0}(x)}{\pi(x)} \geq \frac{1}{2}.$ 

For any fixed (but small) $\epsilon >0$, we have the following inclusion of sets
$$ \{p\leq x: p \nmid 2N_1N_2,\  a(tp^2)b(tp^2)<0\} \supseteq S_{\frac{1}{4\epsilon^2}}([\epsilon,1],[-1,-\epsilon])(x) \cup S_{\frac{1}{4\epsilon^2}}([-1,-\epsilon],[\epsilon,1])(x),$$
where $S_a(I_1, I_2)(x):= \{p \in S(I_1,I_2)(x) : p>a \},$ for any $a \in \R^{+}$.
Hence, we have
$$  \pi_{<0}(x) + \pi \left(\frac{1}{4\epsilon^2}\right) \geq \# S([\epsilon,1],[-1,-\epsilon])(x) + \# S([-1,-\epsilon],[\epsilon,1])(x).$$
Now divide the above inequality by $\pi(x)$
$$ \frac{\pi_{<0}(x)}{\pi(x)} + \frac{\pi\left(\frac{1}{4\epsilon^2}\right) }{\pi(x)} 
\geq \frac{\# S([\epsilon,1],[-1-\epsilon])(x) + \# S([-1-\epsilon],[\epsilon,1])(x)}{\pi(x)}.$$
The term $\frac{\pi(\frac{1}{4\epsilon^2})}{\pi(x)}$ tends to zero as $x\ra \infty$ as $\pi(\frac{1}{4\epsilon^2})$ is finite.
By Conjecture~\ref{Sato-Tate-conj}, we have
$$ \frac{\# S([\epsilon,1],[-1,-\epsilon])(x) + \# S([-1,-\epsilon],[\epsilon,1])(x)  }{\pi(x)}  \lra 2. \mu_{\mrm{ST}}([\epsilon, 1])\mu_{\mrm{ST}}([-1,-\epsilon]) \quad\ \mrm{as}\quad 	x \ra \infty. $$ 
This implies that 
\begin{equation}
\label{key-inequality}
\underset{x \ra \infty}{\mrm{lim\ inf}} \ \frac{\pi_{<0}(x)}{\pi(x)} \geq 2. \mu_{\mrm{ST}}([\epsilon, 1])\mu_{\mrm{ST}}([-1,-\epsilon]),
\end{equation}
where $\pi_{<0}(x)=\# \{p\leq x:p \nmid2N_1N_2,\ a(tp^2)b(tp^2)<0\}$ by definition. Since the inequality~\eqref{key-inequality} holds for all $\epsilon>0$, we have that 
$$ \underset{x \ra \infty}{\mrm{lim\  inf}} \ \frac{\pi_{<0}(x)}{\pi(x)} \geq \mu_{\mrm{ST}}([0,1]) = \frac{1}{2}.$$
A similarly proof  shows that
$\underset{x \ra \infty}{\mrm{lim\ inf}} \ \frac{\pi_{\leq 0}(x)}{\pi(x)} \geq \frac{1}{2}.$
% A similar argument yields $\underset{x \ra \infty}{\mrm{lim\ inf}} \frac{\pi_{\leq 0}(x)}{pi(x)} \geq \frac{1}{2}$. 
Since $\pi_{> 0}(x) = \pi(x) - \pi_{\leq 0}(x)$, we have that
$\underset{x \ra \infty}{\mrm{lim\ sup}} \ \frac{\pi_{> 0}(x)}{\pi(x)} \leq \frac{1}{2}.$
%Thus showing that $\underset{x \ra \infty}{\mrm{lim}} \ \frac{\pi_{> 0}(x)}{\pi(x)} = \frac{1}{2}.$
%By Lemma~\ref{lem-main-thm} and by Proposition~\ref{prop-main-thm}, we have
%$$ \frac{1}{2} \leq \underset{x \ra \infty}{\mrm{lim\  inf}} \ \frac{\pi_{< 0}(x)}{\pi(x)} \leq \underset{x \ra \infty}{\mrm{lim\  sup}} \ \frac{\pi_{< 0}(x)}{\pi(x)} \leq   \frac{1}{2},$$
% where the last inequality follows from the inclusion 
% $ \{p\leq x: a(tp^2)<0\} \subseteq \{ p \leq x:  B_1(p) \in [0,1] \}.$
%$\pi_{< 0} (x) = \pi(x) - \pi_{\geq 0}(x)$ and by~\eqref{limsup_inequality}. 
Hence, the limit  $ \underset{x \ra \infty}{\mrm{lim}} \frac{\pi_{<0}(x)}{\pi(x)}$ exists 
and is equal to $\frac{1}{2}$. Therefore, the natural density of the set $\mathbb{P}_{<0}$ is $\frac{1}{2}$. A similar proof works for the sets $\mathbb{P}_{ \leq 0}$, $\mathbb{P}_{>0}$, $\mathbb{P}_{\geq 0}$. As a consequence, we see that $d(\mathbb{P}_{=0})=0$.
\end{proof}
\begin{remark}
\label{generaltcase}
In the general case, i.e., $a(t)b(t) \neq 0$, we have to state the theorem with the Fourier coefficients $\frac{a(tp^2)}{a(t)}$ instead of $a(tp^2)$.
\end{remark}

Now, we  state the equi-distribution result for the  product of Fourier coefficients of two distinct half-integral weight cuspidal eigenforms, 
by assuming pair Sato-Tate conjecture for their corresponding Shimura lifts.

\begin{thm}
Assume the hypothesis of Theorem~\ref{main-thm-2}. For any two sub-intervals $I_1 \subseteq [-1,1], I_2 \subseteq [-1,1]$, we have 
$$ d(S(I_1,I_2)) = \underset{x \ra \infty}{\mathrm{lim}}  \ \frac{\#S(I_1,I_2)(x)}{\pi(x)} = \mu_{\mrm{ST}}(I_1) \mu_{\mrm{ST}}(I_2),$$
where 
\begin{align*}
S(I_1,I_2)     &=  \left\lbrace p \in \mathbb{P}: p\nmid 4N_1N_2, \frac{a(tp^2)}{2p^{k_1-\frac{1}{2}}} \in I_1, \frac{b(tp^2)}{2p^{k_2-\frac{1}{2}}} \in I_2 \right\rbrace \\
S(I_1,I_2)(x) &=  \left\lbrace p\leq x: p\in S(I_1,I_2) \right\rbrace.
\end{align*}
\end{thm}
\begin{proof}
The proof of this theorem is similar to the proof of Theorem~\ref{main-thm-2}. This is because,
the set of primes $p$ for which $\frac{a(tp^2)}{2p^{k_1-\frac{1}{2}}}  \in I_1$ and
$\frac{A(p)}{2p^{k_1-\frac{1}{2}}} \not \in I_1$ 
are of density zero. Similarly, for the cuspidal eigenform $g$ with the interval $I_2$. The proof of these statements is a generalization
of the proof of Proposition~\ref{keyproposition1}. 
\end{proof}

% We finish this article with a remark. In this article, we assume that $f,g$ are with trivial nebentypus. However, the question make sense even for $f,g$ with non-trivial quadratic characters. 
% The author does not have any idea about how to show that the characters are equal, once the corresponding Shimura lifts are equal. If one can show this, then the rest of the argument should be the same.

\section{Acknowledgements}
The author would like to thank Prof. B. Ramakrishna and Dr. Soma Purkait for their suggestions on the content of this paper.

\bibliographystyle{amsalpha}

\begin{thebibliography}{abcde999}

\bibitem[AIW15]{AIW15}
            Arias-de-Reyna, Sara; Inam, Ilker; Wiese, Gabor. 
           On conjectures of Sato-Tate and Bruinier-Kohnen. 
           Ramanujan J. 36 (2015), no. 3, 455--481.

	
\bibitem[BGHT11]{BGHT11} 
            Barnet-Lamb, Tom; Geraghty, David; Harris, Michael; Taylor, Richard. 
            A family of Calabi-Yau varieties and potential automorphy II. 
            Publ. Res. Inst. Math. Sci. 47 (2011), no. 1, 29--98.
            
\bibitem[CR94]{CR94}
           Chakraborty, K.; Ramakrishnan, B. 
           A note on Hecke eigenforms. 
           Arch. Math. (Basel) 63 (1994), no. 6, 509--516.
            
\bibitem[GHS09]{GHS09}
           Ganguly, Satadal; Hoffstein, Jeffrey; Sengupta, Jyoti. 
           Determining modular forms on ${\SL}_2(\Z)$ by central values of convolution $L$-functions. 
           Math. Ann. 345 (2009), no. 4, 843--857.
           
\bibitem[GKR15]{GKR15}           
          Gun, Sanoli; Kohnen, Winfried; Rath, Purusottam. 
          Simultaneous sign change of Fourier-coefficients of two cusp forms. 
          Arch. Math. (Basel) 105 (2015), no. 5, 413--424.             
            
              
              
              
\bibitem[Koh82]{Koh82} 
         Kohnen, Winfried. 
         Newforms of half-integral weight.  
         J. Reine Angew. Math. 333 (1982), 32--72.
         
\bibitem[KLSW10]{KLSW10}          
         Kowalski, E.; Lau, Y.-K.; Soundararajan, K.; Wu, J. 
         On modular signs. 
         Math. Proc. Cambridge Philos. Soc. 149 (2010), no. 3, 389--411.

\bibitem[Kum15]{Kum15}
          Kumar, Narasimha. 
          Remarks on $q$-exponents of generalized modular functions. 
          Funct. Approx. Comment. Math. 53 (2015), no. 2, 177--188.
          
\bibitem[LR97]{LR97}
        Luo, Wenzhi; Ramakrishnan, Dinakar. 
       Determination of modular forms by twists of critical $L$-values. 
       Invent. Math. 130 (1997), no. 2, 371--398. 
       
\bibitem[MRV90]{MRV90}
        Manickam, M.; Ramakrishnan, B.; Vasudevan, T. C. 
        On the theory of newforms of half-integral weight. 
        J. Number Theory 34 (1990), no. 2, 210--224.       
       
       
\bibitem[Mat12]{Mat12}
        Matom\"aki, Kaisa.
        On signs of Fourier coefficients of cusp forms.
        Mathematical Proceedings of the Cambridge 
        Philosophical Society 152 (2012), no. 02, 207--222.

% \bibitem[Shi73]{Shi73} 
%         Shimura, Goro.
%         On Modular forms of half integral weight.
%         Ann. of Math. (2) 97 (1973), 440--481.          
              
% \bibitem[Pur14]{Pur14}
%         Purkait, Soma. 
%         Hecke operators in half-integral weight. 
%         J. Th\'eor. Nombres Bordeaux 26 (2014), no. 1, 233--251.
        
% \bibitem[Ram94]{Ram94}        
%        Ramakrishnan, Dinakar. 
%        A refinement of the strong multiplicity one theorem for ${\rm GL}(2)$. 
%        Appendix to: "$l$-adic representations associated to modular forms over imaginary quadratic fields. II'' 
%        [Invent. Math. 116 (1994), no. 1-3, 619–643; MR1253207 (95h:11050a)] by R. Taylor. 
%        Invent. Math. 116 (1994), no. 1-3, 645--649.        

% \bibitem[Wal81]{Wal81} 
%         Waldspurger, Jean-Loup.  
%         Sur les coefficients de Fourier des formes modulaires de poids demi-entier. 
%         J. Math. Pures Appl. (9) 60 (1981), no. 4, 375--484.
\end{thebibliography}

\end{document}